\documentclass[11pt]{article}
\usepackage{latexsym,amsmath,amssymb}
\usepackage{amsmath}
\usepackage{subcaption}
\usepackage{graphicx}
\usepackage{amsfonts}
\usepackage{multirow}
\usepackage{epsfig,epsf,psfrag}
\newtheorem{theorem}{Theorem}[section]
\newtheorem{corollary}[theorem]{Corollary}
\newtheorem{lemma}[theorem]{Lemma}
\newtheorem{remark}{Remark}
\newenvironment{proof}[1][Proof]{\begin{trivlist}
\item[\hskip \labelsep {\bfseries #1}]}{\end{trivlist}}

\usepackage{comment}

\newcommand\nc{\newcommand}

\newcommand{\be}{\begin{equation}}
\newcommand{\ee}{\end{equation}}
\newcommand{\ba}{\begin{aligned}}
\newcommand{\ea}{\end{aligned}}
\newcommand{\bea}{\begin{eqnarray}}
\newcommand{\eea}{\end{eqnarray}}

\newcommand{\bA}{\mathbf{A}}
\newcommand{\bE}{\mathbf{E}}
\newcommand{\bH}{\mathbf{H}}
\newcommand{\bJ}{\mathbf{J}}
\newcommand{\bM}{\mathbf{M}}

\nc\ex{E_x}
\nc\ey{E_y}
\nc\ez{E_z}

\nc\hx{H_x}
\nc\hy{H_y}
\nc\hz{H_z}

\nc\px[1]{\frac{\partial #1}{\partial x}}
\nc\py[1]{\frac{\partial #1}{\partial y}}

\nc\br{{\mathbf{r}}}
\nc\curl{\nabla\times}
\nc\dive{\nabla\cdot}
\nc\pt{\frac{\partial}{\partial t}}
\nc\pet{\frac{\partial \bE}{\partial t}}
\nc\pht{\frac{\partial \bH}{\partial t}}

\title{Second Kind Integral Equation Formulation for the Mode Calculation of
  Optical Waveguides}

\author{Jun Lai
\thanks{Courant Institute of Mathematical Sciences, New York University, New York, New York, 10012
({\tt lai@cims.nyu.edu}).}
\and
Shidong Jiang
\thanks{Department of Mathematical Sciences, New Jersey Institute of
Technology, Newark, New Jersey 07102
({\tt shidong.jiang@njit.edu}).}
}
\begin{document}

\maketitle

\begin{abstract}
We present a second kind integral equation (SKIE) formulation
for calculating the electromagnetic modes of optical waveguides,
where the unknowns are only on material interfaces. The resulting numerical
algorithm can handle optical waveguides with
a large number of inclusions of arbitrary irregular cross section.
It is capable of finding the bound, leaky, and complex modes
for optical fibers and waveguides including photonic crystal fibers (PCF),
dielectric fibers and waveguides. Most importantly, the formulation is well
conditioned even in the case of nonsmooth geometries. Our method is highly accurate
and thus can be used to calculate the propagation loss of the electromagnetic modes
accurately, which provides the photonics industry a reliable tool for the design
of more compact and efficient photonic devices. We illustrate and validate
the performance of our method through extensive numerical studies
and by comparison with semi-analytical results and previously published results.
\end{abstract}
\begin{keywords}
  Mode calculation, optical waveguide, optical fiber, hybrid mode,
  second kind integral equation formulation.
\end{keywords}

\pagestyle{myheadings}
\thispagestyle{plain}
\markboth{J. Lai and S. Jiang}
{SKIE Formulation for the Mode Calculation of Optical Waveguides}
\section{Introduction}
Optical fibers and waveguides are important building blocks of many photonic
devices and systems in telecommunication, data transfer and processing, and
optical computing. Indeed, most photonic devices consist of approximately
straight waveguides as input and output channels with complicated functional
structures between the two. Two main mechanisms by which the electromagnetic
wave can be confined in optical fibers or waveguides are total internal
reflection and photonic band gap guidance \cite{okamoto,book2}. Generally speaking,
when the refractive index of the core is greater than that of the surrounding
material, the light is confined in the core by total internal reflection;
when the (hollow) core has a smaller refractive index, confinement can be achieved
through photonic band gap guidance. In both cases, the propagating electromagnetic
modes of optical fibers and waveguides depend on physical parameters such as
the input light wavelength, refractive indices, and the geometry of the cross
section of fibers and waveguides. To reduce the cost of designing new photonic
devices, accurate and efficient simulation tools are in high demand in 
integrated photonics industry. The first step in the photonics simulation
is to compute a complete set of propagating modes accurately and efficiently
for optical fibers or waveguides.

There has been extensive research on the mode calculation of optical fibers
and waveguides and various numerical methods have been developed. These include
the effective index method \cite{marcatili,effectiveindex}, the plane wave
expansion method \cite{planewave1,planewave2,mitmpb}, the multipole expansion
method \cite{multipole1,multipole2,multipole3,multipole4,multipole5,tpwhite1,tpwhite2},
finite difference methods \cite{fd1,fd2,fd3,yee}, finite element methods
\cite{fem1,fem2,fem3,fem4,fem5,fem6,fem7}, boundary integral methods
\cite{bem1,bem2,bem3,bem4,leslie1,leslie2,srep,lu1,lu2,bie6,pone}, etc.
Here we do not intend to
review these methods in great detail, but note that the effective index method
is generally of low order making it difficult to calculate the propagation constant
to high accuracy; the plane wave expansion method implies an infinite periodic
medium; the multipole expansion method requires that each core
be of circular shape and that the cores be well separated from each other; finite
difference and finite element methods requires a volume discretization of 
the cross section in a truncated computational domain with some artificial
boundary conditions or perfectly matched layers imposed on or near
the boundary of the truncated domain. When optical fibers and waveguides consist
of many cores of arbitrary shape, these methods often need excessively large
amount of computing resource in order to accurately calculate the imaginary part of
the propagation constant, which is related to the propagation loss of the
electromagnetic modes and thus of fundamental importance for the design
purpose.

On the other hand, boundary integral methods represent the electromagnetic
fields via layer potentials which satisfy the underlying partial differential
equations automatically. One then derives a set of integral equations
through the matching of boundary conditions with the unknowns only on the
material interfaces. Thus the dimension of the problem is reduced by one and
complex geometries can be handled relatively easily. Among the aforementioned
work on boundary integral methods, \cite{leslie1,srep,lu1,pone} present
numerical examples with high accuracy for smooth geometries. In
\cite{srep} and \cite{pone}, the field components $E_z$ and $H_z$ (with $z$-axis
the longitudinal direction of the waveguide) are represented
via four distinct {\it single} layer potentials and the resulting system is
a mixture of first kind and singular integral equations; both authors apply
 the circular case as a preconditioner
to obtain a well conditioned system for smooth
boundaries. In \cite{leslie1}, $E_z$ and $H_z$ are represented via a proper
linear combination of {\it single} and {\it double} layer potentials in such
a way that the hypersingular terms are cancelled out. The resulting system
still contains the tangential derivatives of the unknown densities and layer
potentials and thus is not of the second kind. In \cite{lu1},
Dirichlet-to-Neumann (DtN) maps for $H_x$ and $H_y$ are
used to construct a system of two integral equations, where
each DtN map is in turn computed by a boundary integral equation with a
hypersingular integral operator and a method in \cite{hypersingular}
is applied to evaluate the DtN map to high accuracy for smooth boundaries.
While these methods are all capable of computing the propagation
constant to high accuracy for smooth cases, it is not straightforward to
extend them to treat nonsmooth cases such as standard dielectric
rectangular waveguides in integrated optics.

\begin{remark}
We would like
to remark here that \cite{srep} has a subsection titled ``Buried Rectangular
Dielectric Waveguide''. In that subsection, the authors approximate
the rectangular waveguide via a smooth super-ellipse and
compute the propagation constant for the super-ellipse. Though Fig.~2 in
\cite{srep} achieves about $9$ digit accuracy for the super-ellipse,
Fig.~3 in \cite{srep} shows only about two digit accuracy for the
propagation constant of the genuine rectangular waveguide which is
regarded as a limit of the super-ellipse.
\end{remark}

In this paper, we construct a system of SKIEs formulation
for the mode calculation of optical waveguides. Our starting point is the dual M\"{u}ller's
formulation \cite{muller}
for the time-harmonic Maxwell's equations in three dimensions.
We then reduce the dimension of the integration domain by one using
the key assumption of the mode calculation
that the dependence on $z$ of the electromagnetic fields and the unknown densities
is of the form $e^{i\beta z}$, where $\beta$ is the propagation constant of the mode.
Accordingly, the layer potentials defined on the surface of the waveguide are 
reduced to the layer potentials defined only on the boundary of the
cross section. And the boundary conditions lead to a system of four SKIEs where
each integral equation consists of a sum of a constant multiple of the identity
operator and several compact operators. All involved integral operators have
logarithmically singular kernels and thus are straightforward to discretize
with high order quadratures. Hence, our formulation leads to a numerical algorithm
which can handle the mode calculation of optical waveguides of arbitrary
geometries (smooth and nonsmooth) accurately and efficiently.

The paper is organized as follows. In Section 2, we develop the SKIE formulation.
We discuss briefly the discretization scheme and numerical algorithm for the mode
finding in Section 3. In Section 4, we illustrate the performance of our scheme
numerically and compare our results with semi-analytical results or
previously published results in \cite{leslie1,lu1,pone} for PCFs
with smooth holes. High accuracy results for dielectric rectangular
waveguides are presented in Section 5. And concluding remarks are contained
in Section 6.
\section{SKIE formulation for the electromagnetic mode calculation}
In this section, we first derive the SKIE formulation for the electromagnetic mode
calculation when the waveguide consists of only one core. The extension to the case
of multiple cores or holes is straightforward.

\begin{figure}[ht]
    \centering
        \includegraphics[width=0.5\linewidth]{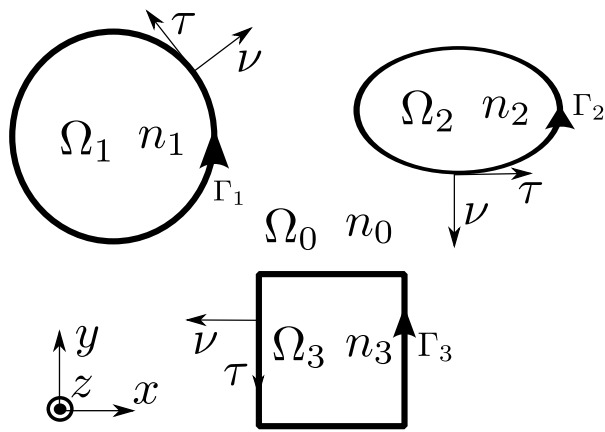}
        \caption{Cross section of an optical waveguide consisting of multiple cores}
    \label{Geometry1}    
\end{figure} 

\subsection{Notation}
We follow standard conventions in the mode calculation literature
and assume that electromagnetic fields are propagated along the $z$-axis, and that
the geometric structure of the optical waveguides is completely determined by its
cross section in the $xy$-plane (see Fig.~\ref{Geometry1} for an illustration). The waveguide consists
of several cores (or holes) denoted by $\Omega_1, \cdots, \Omega_N$ with refractive
indices $n_1, \cdots, n_N$, respectively. The surroundings of these cores or the cladding 
is denoted by $\Omega_0$ with refractive index $n_0$. The boundary of each core
is denoted by $\Gamma_i$ with $\nu$ the unit outward normal vector and $\tau$ the unit
tangential vector, respectively. We denote points in $\mathbb{R}^2$ by $P$ and $Q$.
\subsection{PDE formulation}
The source free Maxwell equations in each homogeneous region are given by:
\bea
\curl \bE&=&-\mu_0 \pht, \label{curle0}\\
\curl \bH&=&\epsilon_0 n^2 \pet, \label{curlh0}\\
\dive \bE&=&0, \label{dive}\\
\dive \bH&=&0, \label{divh}
\eea
where $n$ is the index of refraction of the region.

Assuming the time dependence of the electromagnetic field is $e^{-i\omega t}$,
i.e., $\bE(x,y,,t)=\bE(x,y,z)e^{-i\omega t}$ and $\bH(x,y,,t)=\bH(x,y,z)e^{-i\omega t}$,
and rescaling the magnetic field $\bH$ by $\sqrt{\frac{\mu_0}{\epsilon_0}}$,
equations \eqref{curle0}-\eqref{curlh0} become
\bea
\curl \bE - ik_v\bH = 0, \label{curle}  \\
\curl \bH + i\frac{k^2}{k_v}\bE = 0, \label{curlh}
\eea
where $k=k_vn$ is the wave number in the region,
$k_v=\omega/c$ is the wave number in vacuum, and
$c=\frac{1}{\sqrt{\epsilon_0\mu_0}}$ is the speed of light in vacuum. 
In the mode calculation, a further important assumption is that the electromagnetic
field takes the following form:
\bea
\bE(x,y,z)&=&\bE(x,y)e^{i\beta z}, \label{emode}\\
\bH(x,y,z)&=&\bH(x,y)e^{i\beta z}. \label{hmode}
\eea
We observe that with this assumption each component of the electromagnetic
fields of the mode satisfies the Helmholtz
equation in two dimensions:
\be\label{eeq} 
\left[\Delta + (k^2-\beta^2)\right] u = 0. 
\ee

On the material interface, the boundary conditions are that the tangential components
of the electromagnetic fields are continuous. That is,
\be
[\nu \times \bE]=0, \qquad [\nu \times \bH]=0,
\ee

where $[\cdot]$ denotes the jump across the material interface. It is clear
that $z$ and $\tau$ are two tangential directions for the waveguide geometry.
Thus the boundary conditions can be written explicitly as follows:
\be\label{modebc}
   [E_z]=0, \qquad [E_\tau]=0, \qquad
   [H_z]=0, \qquad [H_\tau]=0.
\ee
Finally, we would like to remark that combining the assumptions \eqref{emode},
\eqref{hmode} and Maxwell's equations, it is straightforward to verify that
there are only two independent components for the electromagnetic fields
in each region. For instance, the components $\ex$, $\ey$, $\hx$, $\hy$ are
completely determined by $\ez$ and $\hz$ via the following relations:
\be\label{hxey}
\begin{bmatrix}\hx\\ \ey\end{bmatrix} 
= \frac{-1}{k^2-\beta^2}\begin{bmatrix} ik^2/k_v & -i\beta \\
-i \beta & i k_v \end{bmatrix} 
 \begin{bmatrix}\py{\ez}\\ \px{\hz}\end{bmatrix}.
\ee
\be\label{hyex}
\begin{bmatrix}\ex\\ \hy\end{bmatrix} 
= \frac{1}{k^2-\beta^2}\begin{bmatrix} i\beta& ik_v \\
i k^2/k_v & i\beta \end{bmatrix} 
 \begin{bmatrix}\px{\ez}\\ \py{\hz}\end{bmatrix}.
\ee
In fact, \cite{leslie1,srep,pone} start from the integral representation of
$\ez$ and $\hz$ to develop the integral equation formulation for the mode
calculation.

%
%
%
\subsection{SKIE formulation}
We first provide an informal description about our construction of the SKIE
formulation for the mode calculation of optical waveguides.
Previous integral equation formulations in \cite{leslie1,srep,pone,lu1}
start from two scalar variables, i.e., two components of the electromagnetic
fields ($\ez$ and $\hz$ in \cite{leslie1,srep,pone} and $\hx$, $\hy$ in \cite{lu1})
and then set up the integral equations through the boundary conditions.
Here we start from the dual M\"{u}ller's representation in \cite{muller}
for the time harmonic electromagnetic fields which leads to an SKIE
formulation for the dielectric interface problems in three dimensions.
As an integral representation for three dimensional problem, the integration domain
is over the boundary $\partial \Omega$ and the unknowns are the surface currents
$\bJ$ and $\bM$ on $\partial \Omega$. For the mode calculation of an
optical waveguide, the boundary is an infinitely long cylinder, i.e.,
$\partial \Omega = \Gamma \times (-\infty,\infty)$ where $\Gamma$
is the boundary of the cross section of the waveguide. Since all electromagnetic
field components have $e^{i\beta z}$ dependence on $z$, it is natural to assume that
the unknown surface currents $\bJ$ and $\bM$ have the same dependence on $z$ as
well. Combining these two factors, we are able to reduce the dimension of the
representation by one and derive an integral representation of $\bE(x,y)$
and $\bH(x,y)$ using layer potentials defined only on $\Gamma$. It is
readily to verify that the boundary conditions lead to a system of SKIEs
for the mode calculation, with the propagation constant $\beta$ appearing
as a nonlinear parameter of the system.

\subsubsection{Reduction of layer potentials in the mode calculation}
The dual M\"{u}ller's representation \cite{muller} assumes that in each region
the time harmonic
electromagnetic field
$\bE(x,y,z)$ and $\bH(x,y,z)$ have the following representation:
\bea
\bE &=& \frac{1}{ik_v}\curl( \curl S^{k}_{\partial \Omega} [\bJ]) -\nabla \times S^{k} _{\partial \Omega} [\bM], \label{mullere0} \\
\bH &=& -\frac{1}{ik_v}\curl( \curl S^{k}_{\partial \Omega}[\bM]) - \frac{k^2}{k_v^2}\nabla \times S^{k} _{\partial \Omega}[\bJ]. \label{mullerh0}
\eea
Here $\partial \Omega$ is the boundary of the three dimensional domain,
$\bJ$ and $\bM$ are the unknown surface electric and magnetic currents, and 
\begin{equation}
  S^{k}_{\partial \Omega}[\bJ]
  (\br)  =\frac{1}{4\pi} \int_{\partial \Omega}  \frac{e^{ik|\br-\br'|}}{|\br-\br'|}
  \bJ(\br')d\br'
\end{equation}
is the single layer potential for the Helmholtz equation in 3D
with similar expression for $S^{k}_{\partial \Omega}[\bM]$. Obviously, 
$S^{k}_{\partial \Omega}[\bJ]$ satisfies the Helmholtz equation in 3D, i.e.,
$(\nabla^2+k^2)S^{k}_{\partial \Omega}[\bJ](\br)=0$ for $\br \in \Omega$.
It is straightforward to verify that \eqref{mullere0}-\eqref{mullerh0}
satisfy Maxwell's equations \eqref{dive}-\eqref{curlh}.

For the waveguide geometry, $\partial \Omega = \Gamma \times (-\infty,\infty)$.
Furthermore, since both $\bE(x,y,z)$ and $\bH(x,y,z)$ depend on $z$ in the form
of $e^{i\beta z}$, it is natural to assume that
the surface currents $\bJ(x,y,z)$ and $\bM(x,y,z)$ have the same
$z$ dependence as well, that is,  
\be\label{jmdef}
\bJ(x,y,z)=\bJ(x,y)e^{i\beta z}, \qquad \bM(x,y,z)=\bM(x,y)e^{i\beta z}.
\ee
We now introduce some notation to be used subsequently. We denote
$k_{\beta} = \sqrt{k^2-\beta^2}$
and the Green's function for the Helmholtz equation \eqref{eeq}
in 2D by $G(P,Q)=\frac{i}{4}H^{(1)}_0(k_{\beta}|P-Q|)$,
where $H^{(1)}_0$ is the zeroth order Hankel function of the first kind
\cite{handbook}. We further denote by $S$ the single layer potential operator
defined by the formula
\be
S[\sigma]=\int_\Gamma G(P,Q) \sigma(Q)ds_Q;
\ee
  $D$ the double layer potential operator defined by the formula
\be
D[\sigma]=\int_\Gamma \frac{\partial G(P,Q)}{\partial \nu(Q)} \sigma(Q)ds_Q;
\ee
and $T$ the anti-double layer potential operator
(for the lack of standard terminology) defined by the formula 
\be
T[\sigma]=\int_\Gamma \frac{\partial G(P,Q)}{\partial \tau(Q)}\sigma(Q)ds_Q.
\ee
The following lemma shows that the Fourier transform of the 3D Helmholtz
Green's function is the 2D Helmholtz Green's function.
\begin{lemma}\label{kernelreduction}
\begin{equation}\label{greeniden}
  \int_{-\infty}^{\infty} \frac{1}{4\pi} \frac{e^{ik|\br-\br'|}}{|\br-\br'|}
  e^{i\beta z'}dz'= G(P,Q)e^{i\beta z},
\end{equation}
where
$P=(x,y)$ and $Q=(x',y')$ are the projections
of $\br=(x,y,z)$ and $\br'=(x',y',z')$ onto the $xy$-plane, respectively, and the wavenumber $k_{\beta}$ in $G$ is chosen to have nonnegative imaginary part.
\end{lemma}

\begin{proof}
Let $|P-Q| = \sqrt{(x-x')^2+(y-y')^2}=a $, and $t = z-z'$. Then
\begin{equation}
  \frac{1}{4\pi} \int_{-\infty}^{\infty}  \frac{e^{ik|\br-\br'|}}{|\br-\br'|} e^{i\beta z'}dz'
  = e^{i\beta z} \frac{1}{4\pi} \int_{-\infty}^{\infty}
  \frac{e^{ik\sqrt{a^2+t^2}}}{\sqrt{a^2+t^2}} e^{-i\beta t}dt
\end{equation}
We now note that the well-known Sommerfeld integral identity (for $y>0$)
\cite{sommerfeld} is given by:
\begin{equation}
  \frac{i}{4}H_0^1(k_s\sqrt{x^2+y^2}) = \frac{1}{4\pi}
  \int_{-\infty}^{\infty} \frac{e^{-\sqrt{\lambda^2-k_s^2}y}}{\sqrt{\lambda^2-k_s^2}}e^{i\lambda x}d\lambda.
\end{equation}
\eqref{greeniden} follows from the substitution
$\lambda = it$, $y=k$, $x=i\beta$ and $k_s = a$
and appropriate contour deformation. 
\end{proof}

The following corollary reduces the single layer potential
$S^{k}_{\partial \Omega}$ in 3D to a single layer potential in 2D in the mode
calculation of an optical waveguide.

\begin{corollary}\label{srepreduction}
Suppose that $\partial \Omega = \Gamma \times (-\infty,\infty)$
and density $\sigma(x,y,z)=\mu(x,y)e^{i\beta z}$. Then
\be
S^{k}_{\partial \Omega}[\sigma](\br)=e^{i\beta z}S[\mu](P).
\ee    
\end{corollary}
\begin{proof}
\be
\ba
S^{k}_{\partial \Omega}[\sigma](\br)&=\int_\Gamma \mu(x',y')ds_\Gamma \int_{-\infty}^{\infty}
\frac{1}{4\pi} \frac{e^{ik|\br-\br'|}}{|\br-\br'|} e^{i\beta z'}dz'\\
&=e^{i\beta z}\int_\Gamma \frac{i}{4}H^{(1)}_0(k_{\beta}|P-Q|)\mu(Q)ds_Q\\
&=e^{i\beta z}S[\mu](P),
\ea
\ee
where the second equality follows from \eqref{greeniden}.
\end{proof}

\subsubsection{Integral representations of the electromagnetic field
  in the mode calculation}
Let $(\hat{i},\hat{j},\hat{k})$ be the basis of the Cartesian coordinate with $\hat{k}$ pointing along the positive $z$ direction. We write the unit tangent vector
in its component form $\tau = \tau_1 \hat{i}+\tau_2\hat{j}$.
Then the outward unit normal vector $\nu$ is given by
$\nu=\nu_1\hat{i}+\nu_2\hat{j} =\tau_2\hat{i}-\tau_1\hat{j}$. Since $\bJ$ and $\bM$
are unknown surface currents and $\tau$, $\hat{k}$ are two locally orthonormal
tangential vectors at a point $(x,y,z)\in \partial \Omega$,
we may write $\bJ(x,y)$ and $\bM(x,y)$ in \eqref{jmdef} for $(x,y)\in \Gamma$
as follows:
\be\label{jmdef2}
\ba
\bJ(x,y) &= J_\tau(x,y)\tau + J_z(x,y) \hat{k} 
         = J_\tau \tau_1 \hat{i} + J_\tau \tau_2 \hat{j} + J_z \hat{k},\\
\bM(x,y) &= M_\tau(x,y) \tau + M_z(x,y) \hat{k}
         = M_\tau \tau_1 \hat{i} + M_\tau \tau_2 \hat{j} + M_z \hat{k}.
\ea\ee
We now write down our integral representations for the electromagnetic field
in each region in the mode calculation of optical waveguides.
\begin{theorem} 
  [Integral representations of the electromagnetic field]\label{integralrep}
  Suppose that the unknown surface currents $\bJ(x,y,z)$ and $\bM(x,y,z)$
  are given in the form of \eqref{jmdef} and \eqref{jmdef2}.
  Then $\bE(x,y,z)$ defined in \eqref{mullere0} has the form
$\bE(x,y,z)=(E_x(x,y),E_y(x,y),E_z(x,y))e^{i\beta z}$, where $E_x$, $E_y$ and $E_z$
are defined via the formulas:
\be\label{mullere}\ba
E_x &= -\frac{1}{ik_v}
\frac{\partial}{\partial x}T[J_\tau]
+\frac{\beta}{k_v}
\frac{\partial}{\partial x}S[J_z] +\frac{k^2}{ik_v}
S[J_\tau\tau_1]
-\frac{\partial}{\partial y}S[M_z]+i\beta S[M_\tau \tau_2],\\
E_y &= -\frac{1}{ik_v}\frac{\partial}{\partial y}T[J_\tau]
+\frac{\beta}{k_v} \frac{\partial}{\partial y}S[J_z] +\frac{k^2}{ik_v}S[J_\tau\tau_2]
+\frac{\partial}{\partial x}S[M_z]-i\beta S[M_\tau \tau_1],\\
E_z &= -\frac{\beta}{k_v} T[J_\tau]
+\frac{(k^2-\beta^2)}{ik_v}
S[J_z]+D[M_\tau].
\ea\ee
Similarly, $\bH(x,y,z)=(H_x(x,y),H_y(x,y),H_z(x,y))e^{i\beta z}$, where
$H_x$, $H_y$ and $H_z$
are defined via the formulas:
\be\label{mullerh}\ba
H_x &= \frac{1}{ik_v}\frac{\partial}{\partial x}T[M_\tau]
-\frac{\beta}{k_v} \frac{\partial}{\partial x}S[M_z] -\frac{k^2}{ik_v}S[M_\tau\tau_1]
-\frac{k^2}{k_v^2}\frac{\partial}{\partial y}S[J_z]
+i\beta \frac{k^2}{k_v^2} S[J_\tau \tau_2],\\
H_y &= \frac{1}{ik_v}\frac{\partial}{\partial y}T[M_\tau]
-\frac{\beta}{k_v} \frac{\partial}{\partial y}S[M_z] -\frac{k^2}{ik_v}S[M_\tau\tau_2]
+\frac{k^2}{k_v^2}\frac{\partial}{\partial x}S[J_z]
-i\beta \frac{k^2}{k_v^2} S[J_\tau \tau_1],\\
H_z &= \frac{\beta}{k_v} T[M_\tau]
-\frac{(k^2-\beta^2)}{ik_v} S[M_z]+\frac{k^2}{k_v^2}D[J_\tau].
\ea\ee
\end{theorem}
\begin{proof}
By Corollary \ref{srepreduction}, we have
\be
S^{k}_{\partial \Omega}[\bJ]=e^{i\beta z}S[\bJ],  \qquad
S^{k}_{\partial \Omega}[\bM]=e^{i\beta z}S[\bM].  
\ee
Thus,
\be\label{sjrep}
S^{k}_{\partial \Omega}[\bJ] = e^{i\beta z}\left(S[J_\tau \tau_1]\hat{i}
+S[J_\tau \tau_2]\hat{j}+S[J_z]\hat{k}\right)
\ee
and similar expression holds for $S^{k}_{\partial \Omega}[\bM]$.
Since
\be
\ba
\curl S^{k}_{\partial \Omega}[\bJ] &= e^{i\beta z} \left\{
\left(\frac{\partial}{\partial y}S[J_z]-i\beta S[J_\tau \tau_2]\right)
\hat{i}
\right.\\
&+
\left(-\frac{\partial}{\partial x}S[J_z]+i\beta S[J_\tau \tau_1]\right)\hat{j}\\
&+\left.
\left(\frac{\partial}{\partial x}S[J_\tau\tau_2]
               -\frac{\partial}{\partial y} S[J_\tau \tau_1]\right)\hat{k}\right\}
\ea
\ee
and
\be
\ba
&\frac{\partial}{\partial x}S[J_\tau\tau_2]
-\frac{\partial}{\partial y} S[J_\tau \tau_1]\\
&\ =-\int_\Gamma \frac{\partial G(P,Q)}{\partial x'} J_\tau(Q)\tau_2(Q)ds_Q
 +\int_\Gamma \frac{\partial G(P,Q)}{\partial y'} J_\tau(Q)\tau_1(Q)ds_Q\\
&\ =-\int_\Gamma \frac{\partial G(P,Q)}{\partial x'} J_\tau(Q)\nu_1(Q)ds_Q
-\int_\Gamma \frac{\partial G(P,Q)}{\partial y'} J_\tau(Q)\nu_2(Q)ds_Q\\
&=-\int_\Gamma \frac{\partial G(P,Q)}{\partial \nu(Q)} J_\tau(Q)ds_Q
=-D[J_\tau](P),
\ea
\ee
we have
\be\label{curlj}
\ba
\curl S^{k}_{\partial \Omega}[\bJ] &= e^{i\beta z} \left\{
\left(\frac{\partial}{\partial y}S[J_z]-i\beta S[J_\tau \tau_2]\right)\hat{i}
\right.\\
&\left.
+\left(-\frac{\partial}{\partial x}S[J_z]+i\beta S[J_\tau \tau_1]\right)\hat{j}
-D[J_\tau]\hat{k}\right\}.
\ea
\ee
Using the identity $\curl(\curl \bA)=\nabla(\dive \bA)-\nabla^2\bA$, we have
\be\label{curl2j}
\ba
\curl(\curl S^{k}_{\partial \Omega}[\bJ]) &= \nabla(\dive S^{k}_{\partial \Omega}[\bJ])
-\nabla^2S^{k}_{\partial \Omega}[\bJ]\\
&= \nabla(\dive S^{k}_{\partial \Omega}[\bJ])+k^2S^{k}_{\partial \Omega}[\bJ]\\
&= \nabla(\dive (e^{i\beta z}S[\bJ])
)+k^2e^{i\beta z}S[\bJ],
\ea
\ee
and
\be
\ba
\dive (e^{i\beta z}S[\bJ]) &= e^{i\beta z}\left(
\frac{\partial}{\partial x}S[J_\tau \tau_1]
+\frac{\partial}{\partial y}S[J_\tau \tau_2]
+i\beta S[J_z]\right)\\
&=e^{i\beta z}\left(-T[J_\tau]+i\beta S[J_z]\right).
\ea
\ee
Thus,
\be\label{curl2j1}
\ba
\nabla(\dive (e^{i\beta z}S[\bJ]) &= e^{i\beta z} \left\{
\left(-\frac{\partial}{\partial x}T[J_\tau]
+i\beta \frac{\partial}{\partial x}S[J_z]\right)\hat{i}\right.\\
&+
\left(-\frac{\partial}{\partial y}T[J_\tau]
+i\beta \frac{\partial}{\partial y}S[J_z]\right)\hat{j}\\
&\left.+
\left(-i\beta T[J_\tau]
-\beta^2 S[J_z]\right)\hat{k}\right\}.
\ea
\ee
Substituting \eqref{sjrep}, \eqref{curlj} (with $\bJ$ replaced
by $\bM$), \eqref{curl2j}, \eqref{curl2j1} into \eqref{mullere0},
we obtain \eqref{mullere}. And \eqref{mullerh} can be obtained similarly.
\end{proof}
\begin{remark}
  Though $\beta$ is assumed to be real in Lemme \ref{kernelreduction},
  we will use \eqref{mullere}-\eqref{mullerh} to represent the electromagnetic
  field even when $\beta$ is complex for the mode calculation. In fact,
  it is straightforward (though a little tedious) to verify
  \eqref{mullere}-\eqref{mullerh} satisfy the relations \eqref{hxey}-\eqref{hyex}
  even when $\beta$ is complex. The hard part is to verify the following
  identities:
\be
\ba
\frac{\partial}{\partial y}T[M_\tau]+\frac{\partial}{\partial x}D[M_\tau]
&=-\nabla^2 S[M_\tau \tau_2] = (k^2-\beta^2)S[M_\tau \tau_2],\\
-\frac{\partial}{\partial x}T[J_\tau]+\frac{\partial}{\partial y}D[J_\tau]
&=\nabla^2 S[J_\tau \tau_1] = -(k^2-\beta^2)S[J_\tau \tau_1],\\
\ea
\ee
and other similar identities. In other words, we might start directly
from the representation \eqref{mullere}-\eqref{mullerh} for the mode calculation
without even mentioning the dual M\"{u}ller's representation, which only serves as
a formal derivation tool. 
\end{remark}
\subsubsection{Main theoretical result}
We are now in a position to derive the SKIE formulation for the mode calculation
of photonic waveguides. Since the boundary conditions for dielectric problems
are that the tangential
components of $\bE$ and $\bH$ must be continuous across the boundary.
We first combine the first two equations in \eqref{mullere} and \eqref{mullerh}
to obtain $E_\tau$ and $H_\tau$. Together with $E_z$ and $H_z$, we list all
four tangential components as follows:
\be\label{mullereh}
\ba
H_z &= \frac{k^2}{k_v^2}D[J_\tau]+\frac{\beta}{k_v} T[M_\tau]
+\frac{i(k^2-\beta^2)}{k_v} S[M_z],\\
H_\tau(P) &= i\beta\frac{k^2}{k_v^2} S[J_\tau(Q) \tau(P)\cdot \nu(Q)](P)
+\frac{k^2}{k_v^2} S_\nu[J_z](P)\\
&+i\frac{k^2}{k_v}S[M_\tau(Q)\tau(P)\cdot \tau(Q)](P)-\frac{i}{k_v}T_\tau[M_\tau](P)
-\frac{\beta}{k_v} S_\tau[M_z](P),\\
E_z &= -\frac{\beta}{k_v} T[J_\tau]
-\frac{i(k^2-\beta^2)}{k_v} S[J_z]+D[M_\tau], \\
E_\tau(P) &= -\frac{ik^2}{k_v}S[J_\tau(Q)\tau(P)\cdot \tau(Q)](P)+\frac{i}{k_v}T_\tau[J_\tau](P)
+\frac{\beta}{k_v} S_\tau[J_z](P)\\
&+i\beta S[M_\tau(Q) \tau(P)\cdot \nu(Q)](P)+ S_\nu[M_z](P).
\ea
\ee
We now scale and nondimensionalize the quantities in the above equation
by multiplying all length quantities with $k_v$. We have
\be\label{mullereh2}
\ba
\tilde{H}_z &= n^2D[J_\tau]
+n_e T[M_\tau]
+i(n^2-n_e^2) S[M_z],\\
-\tilde{H}_\tau(P) &= -in^2n_e S[J_\tau(Q) \tau(P)\cdot \nu(Q)](P)
-n^2 S_\nu[J_z](P)\\
&-in^2S[M_\tau(Q)\tau(P)\cdot \tau(Q)](P)+iT_\tau[M_\tau](P)
+n_e S_\tau[M_z](P),\\
\tilde{E}_z &= -n_e T[J_\tau]
-i(n^2-n_e^2) S[J_z]+D[M_\tau], \\
-\tilde{E}_\tau(P) &= in^2S[J_\tau(Q)\tau(P)\cdot \tau(Q)](P)-iT_\tau[J_\tau](P)
-n_e S_\tau[J_z](P)\\
&-in_e S[M_\tau(Q) \tau(P)\cdot \nu(Q)](P)
- S_\nu[M_z](P),
\ea
\ee
where $n_e=\beta/ k_v$ is the effective index.

For each region $\Omega_i$ ($i=0, 1$ for now),
we denote the single, double, and anti-double layer
potential operators in that region by $S_i$, $D_i$, and $T_i$, respectively.
We now summarize our main theoretical result in the following theorem.
\begin{theorem}
  [SKIE formulation for the mode calculation] \label{skieformulation}
  Suppose that $\bE(x,y)$ and $\bH(x,y)$ in each region $\Omega_i$ ($i=0,1$)
  are represented via \eqref{mullere}-\eqref{mullerh} with
  $k$, $S$, $D$, $T$ replaced by $k_i$, $S_i$, $D_i$, $T_i$, respectively.
  Then the densities $\bJ(x,y)$ and $\bM(x,y)$ defined in \eqref{jmdef2}
  satisfy the following (nondimensionalized) system of second kind
  integral equations:
\be\label{skie1}
(D+A(n_e))x=0,
\ee
where $x=[J_\tau\ J_z\ M_\tau\  M_z]^T$ is a $4\times 1$ block vector,
$D$ and $A$ are $4\times 4$ block matrices defined by the formulas
\be\label{skie2}
D = \begin{bmatrix} &\frac{n_0^2+n_1^2}{2}
  & & & &\\
  & & \frac{n_0^2+n_1^2}{2} & & & \\
  & & &1 & &\\
  & & & & 1 &\end{bmatrix},
\quad
A(n_e) = \begin{bmatrix} &A_{11}& 0 &A_{13}&A_{14}&\\
  &A_{21}&A_{22}&A_{23}&A_{24}&\\
  &-A_{13}&-A_{14}&A_{33}& 0 &\\
  &-A_{23}&-A_{24}&A_{43}&A_{44}&\end{bmatrix}.
\ee
Here the nonzero entries of $A$ are given by the following formulas:
\be\label{skie3}
\ba
A_{11} &= n_0^2D_0-n_1^2D_1, \qquad A_{22} = -n_0^2S_{0,\nu}+n_1^2S_{1,\nu},\\
A_{33} &= D_0-D_1, \qquad A_{44} = -S_{0,\nu}+S_{1,\nu}, \\
A_{21} &= -in_e\left(n_0^2S_0-n_1^2S_1\right)[\tau(P)\cdot \nu(Q) \cdot],\\
A_{43} &= -in_e(S_0-S_1)[\tau(P)\cdot \nu(Q)\cdot],\\
A_{13} &= n_e (T_0-T_1), \qquad A_{24} = n_e\left(S_{0,\tau}-S_{1,\tau}\right),\\
A_{14} &= i\left((n_0^2-n_e^2)S_0 -(n_1^2-n_e^2)S_1\right),\\
A_{23} &= i(T_{0,\tau}-T_{1,\tau})
-i(n_0^2S_0-n_1^2S_1)[\tau(P)\cdot \tau(Q) \cdot].
\ea
\ee
The effective index $n_e=\beta/k_v$ of the electromagnetic mode is
a complex number for which the above linear system has a nontrivial
solution.
\end{theorem}
\begin{proof}
  We first obtain the nondimensionalized tangential components 
  in $\Omega_i$ ($i=0,1$) by replacing $n$, $S$, $D$, $T$ in \eqref{mullereh2}
  with $n_i$, $S_i$, $D_i$, $T_i$, respectively. We then substitute
  the resulting expressions into the boundary conditions \eqref{modebc}
  and observe that the principal part of the linear system is exactly the
  matrix $A$ defined in \eqref{skie2}-\eqref{skie3}. 

  It is easy to see now that only $D$ and $S_\nu$ have nonzero jumps
  when the target point $P$ approaches the boundary and their jump relations lead
  to the diagonal matrix $D$ in \eqref{skie2} (see, for example, \cite{kress}).
  Furthermore, all entries $A_{ij}$ are compact operators due to following reasons.
  First, $S$, and the principal part of $D$ and $S_\nu$ are compact from standard
  potential theory \cite{kress}. Second, $T_0-T_1$ and $S_{0,\tau}-S_{1,\tau}$ are
  compact since their kernels are only logarithmically singular due to the cancellation of more singular terms. Third, $T_{0,\tau}-T_{1,\tau}$ is also compact since
  the hypersingular terms in difference kernel cancel out. Thus, the resulting
  system is of the second kind.
\end{proof}
\subsection{Extension to the waveguide with multiple holes}
When the waveguide consists of multiple holes, say, $\Omega_1, \cdots, \Omega_N$,
we simply represent $\bE(x,y)$ and $\bH(x,y)$ in each region $\Omega_i$
($i=1,\cdots,N$) via \eqref{mullere}-\eqref{mullerh} with
$k$, $S$, $D$, $T$, $\Gamma$ replaced by $k_i$, $S_i$, $D_i$, $T_i$, $\Gamma_i$,
respectively. For the exterior domain $\Omega_0$, we have similar
representations for $\bE(x,y)$ and $\bH(x,y)$, except that the boundary
$\Gamma$ is replaced by $\Gamma_0 = \Gamma_1 \cup \cdots \cup \Gamma_N$.
The boundary conditions on each $\Gamma_i$ lead to a $4N\times 4N$ block system.
We note that all $4\times 4$ diagonal blocks are similar to
\eqref{skie1}-\eqref{skie3} and that the entries in off-diagonal blocks
are all compact operators since their kernels are smooth due to the fact that
the target point $P$ is bounded away from the source curve. Thus, the system
is of the second kind.
\section{Discretization and numerical algorithms}
As pointed out in Theorem \ref{skieformulation}, all the integral operators in
\eqref{skie1}-\eqref{skie3} have kernels with logarithmic singularity. When
the boundary curves are smooth, there are many high-order quadrature rules
based on local modifications of the trapezoidal rule \cite{alpert,kapur}
or kernel splitting method \cite{sidi,kress_boundary_1991}. When the boundary
curves contain corners, recent developments in
\cite{helsing1,helsing2,helsing3,bremer1,bremer2,bremer3} treat such cases with
high order quadratures very efficiently. In order to achieve optimal efficiency
in discretization, one should treat smooth and nonsmooth cases separately and
implement the aforementioned schemes (say, \cite{alpert,bremer3}) accordingly.

Here we adopt a simpler scheme to discretize both smooth and nonsmooth cases.
We divide the boundary curve into $N_c$ chunks as follows. If the curve is smooth,
the chunks are of equal length in the parameter space. If the curve contains corners,
we will have finer and finer dyadic chunks toward the corner. On each chunk, the
unknown densities are approximated via a $p-1$th order polynomial and the set of
collocation points are the collection of the images of $p$
shifted and scaled Gauss-Legendre nodes on each chunk. When the target points
and the source points lie on the same chunk, we apply a precomputed generalized
Gaussian quadrature (see, for example, \cite{ggq1,ggq2,ggq3}) to discretize the
associated (logarithmically) singular integrals; and when the target points lie
outside the integration chunk, we simply use an adaptive Gaussian integrator to
compute the corresponding matrix entries. The total number of discretization points
is $N=N_c p$, where $N_c$ is the number of chunks, and the size of the resulting matrix $M(n_e)$ is $4N\times 4N$.

We apply the same method as in \cite{leslie1} to find the effective index $n_e$
such that $M(n_e)$ has a nontrivial nullspace. That is, we use M\"{u}ller's method
\cite{dmuller} to find the root of the function
\be\label{mullerfunction}
f(n_e) = \frac{1}{u^TM^{-1}(n_e)v},
\ee
where $u$ and $v$ are two fixed random column vector of length $4N$.

\section{Numerical examples for smooth boundaries}
\subsection{Example 1: optical fiber with a circular core}
For our first example, we consider the silica optical fiber consisting of a single circular core of radius
$50\mu m$.  At the incident wavelength $\lambda = 1500 nm$, the refractive index of the core is $n_1=1.4475$ while
that of the cladding is $n_0=1.444$. 
The effective index can be calculated independently by the multipole
method \cite{Lai:14} and our SKIE formulation. For this simple case,
the multipole method is semi-analytical and one only needs to solve
a $4\times 4$ system to find the effective index every time. Thus 
its results can be taken as the reference values for comparison purpose.
Table \ref{tab:NumericalResult1} presents the first five modes obtained
by our formulation with $50$ discretization points of the boundary and by
the multipole method, which shows IEEE double precision agreement of these two results.
We would like to remark here that our algorithm already achieves $13$ digit
accuracy with only $20$ discretization points. 
\begin{table}[htbp]
\begin{center}
\begin{tabular}{|c|c|c|c|} \hline
   \multicolumn{2}{|c|}{$n_e$ by SKIE}& \multicolumn{2}{|c|}{Reference value}       \\\hline
  Real& Imaginary & Real& Imaginary \\\hline
 1.444873245456804 &  0 & 1.444873245456804 & 0 \\\hline
 1.445573321563491  & 0 & 1.445573321563491 & 0 \\\hline
 1.445671696122978 & 0 & 1.445671696122978 & 0\\\hline
 1.446222363089593 & 0 & 1.446222363089593 & 0\\\hline
 1.447115413503111 & 0 & 1.447115413503111 & 0\\\hline
\end{tabular}
\end{center}
\caption{Effective index of Example 1.  The first column lists the values obtained by the SKIE formulation
  with $50$ discretization points of the circular boundary; while the second column lists the values obtained
  by the multipole method in \cite{Lai:14}. }
\label{tab:NumericalResult1}
\end{table}

\subsection{Example 2: PCFs with a hexagon ring and its perturbation}
\begin{figure}[ht]
    \centering
    \begin{subfigure}[t]{.40\linewidth}
        \centering
        \includegraphics[width=1.0\linewidth]{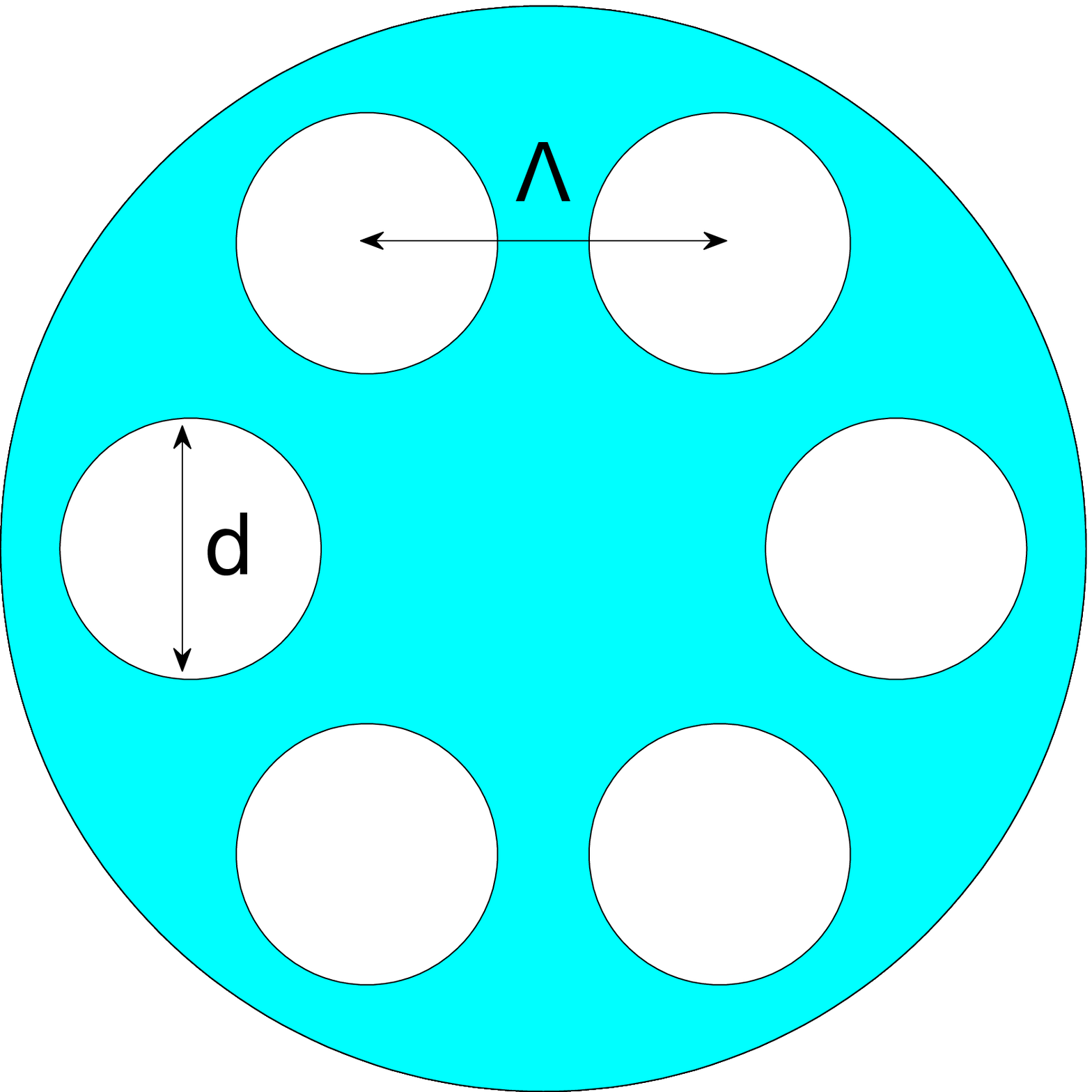}
        \caption{}
    \end{subfigure}
    \quad
    \begin{subfigure}[t]{.40\linewidth}
        \centering
        \includegraphics[width=1.0\linewidth]{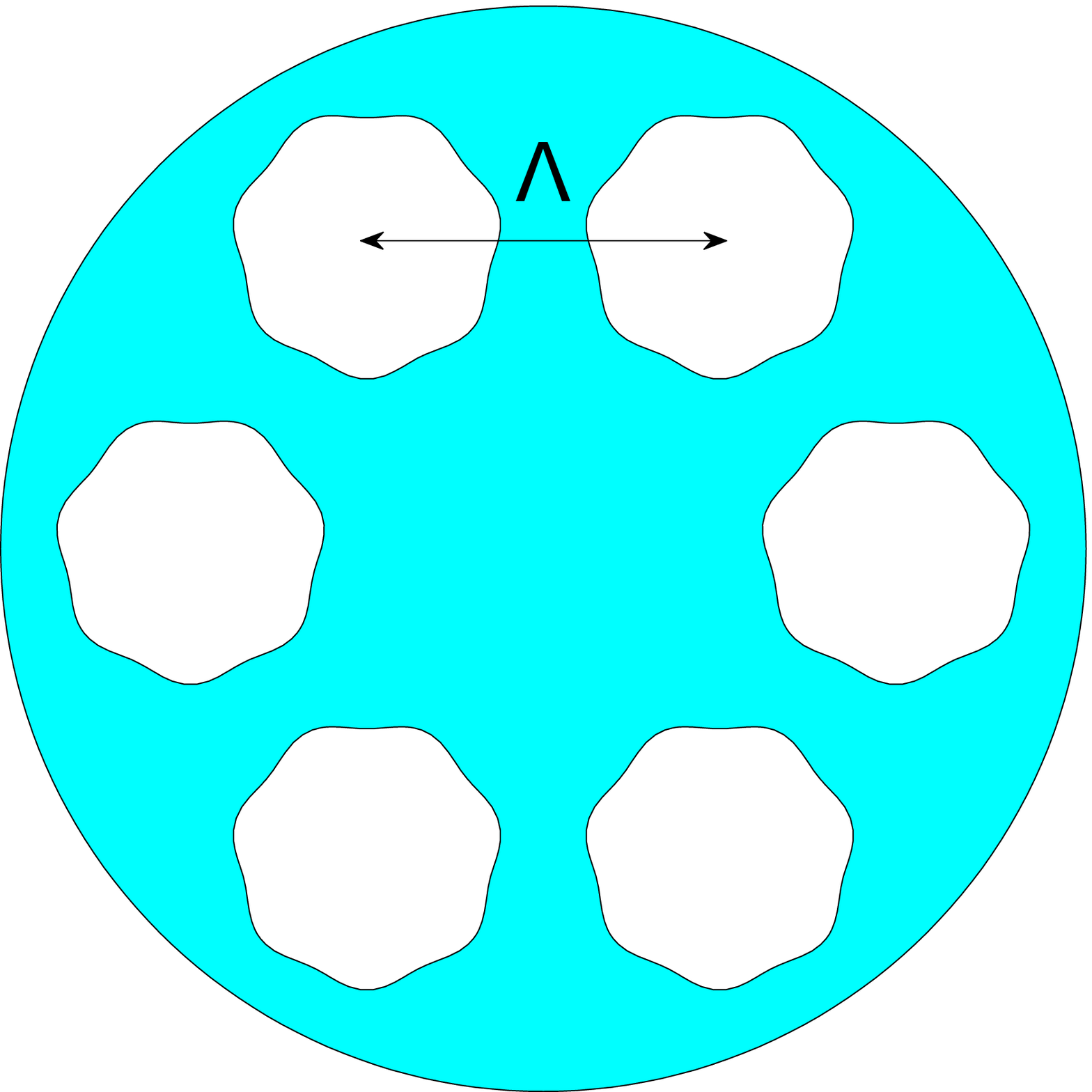}
        \caption{}
    \end{subfigure}
\caption{Geometry of the hexagon ring PCF (left panel) and its perturbation (right panel).}
    \label{hexagonring}
\end{figure}
In this example, we consider two PCFs studied in \cite{leslie1} and \cite{lu1}.
The first PCF consists of a hexagon ring of circular air holes
as shown in Figure \ref{hexagonring}(a). The six holes are equally spaced
along the hexagon with hole diameter $d=5\mu m$ and hole pitch $\Lambda=6.75\mu m$.
The refractive index of the glass matrix is assumed to be $n_0=1.45$ and that of the air
hole is $1$. The incident wavelength is $1450 nm $. The second PCF shown in Fig. \ref{hexagonring}(b)
is the perturbation of the first one, where the holes have the
parametrization: 
\begin{equation}\label{4.2}
||r(\theta)-c_i|| = \frac{d}{2}(1+h\sin{7\theta})
\end{equation} 
Here $r(\theta)$ is the boundary of the $i$-th perturbed circle centered at $c_i$ and the perturbation
level $h$ goes from $1\%$ to $6\%$. 

We discretize each boundary curve by $100$ points. Table 2 lists the
results obtained by our algorithm and the corresponding results in \cite{leslie1} for the first PCF.
We observe that our algorithm is able to recover the first $10$ digits in \cite{leslie1}
for all modes.  
\begin{table}[htbp]
\begin{center}
\begin{tabular}{|c|c|c|c|} \hline
  \multicolumn{2}{|c|}{$n_e$}      &  \multicolumn{2}{|c|}{$n_e$ in
    \cite{leslie1}} \\\hline
   Real& Imaginary &  Real& Imaginary \\\hline
   1.44539523214929 & 3.19452506E-8  & 1.445395232&3.1945E-8 \\\hline
   1.43858364729142 & 5.310787285E-7& 1.438583647 & 5.3108E-7\\\hline
   1.43844483196668 & 9.730851491E-7& 1.438444832 & 9.7308E-7\\\hline
   1.43836493417887 & 1.4164759939E-6& 1.438364934 & 1.41647E-6\\\hline
   1.43040909603339 & 2.15661649916E-5& 1.430409096 & 2.15662E-5\\\hline
   1.42995686266711&  1.59153224394E-5& 1.429956863 & 1.59153E-5\\\hline
   1.42924806251945 & 8.7312643348E-6 & 1.429248062 & 8.7313E-6\\\hline
\end{tabular}
\end{center}
\caption{Effective index of the PCF shown in Fig. \ref{hexagonring}(a).}
\label{tab:NumericalResult2}
\end{table}

Table \ref{tab:NumericalResult3} lists the results for the perturbed PCF.
Compared with Table 2 in \cite{leslie1}, we observe that the two results again
have at least $10$ digit agreement with each other.
\begin{table}[htbp]
\begin{center}
\begin{tabular}{|c|c|c|c|c|} \hline
  & \multicolumn{2}{|c|}{$n_e$ for mode 1}      &  \multicolumn{2}{|c|}{$n_e$ for mode 2} \\\hline
 h & Real& Imaginary &  Real& Imaginary \\\hline
1\%  & 1.44539377438717 &3.17269747E-8   &
       1.44539377640333  &      3.17256375E-8 \\\hline
2\% &  1.44538941232987  &3.10822988E-8   &
       1.44538942033048  &      3.10770686E-8 \\\hline
3\% &  1.44538217911076 & 3.00407424E-8   & 
       1.44538219687975      &  3.00293984E-8 \\\hline
4\%  & 1.44537215920563 & 2.86293848E-8 &
       1.44537212816737   &      2.86485611E-8 \\\hline
5\%  & 1.44535933076758 &2.69649433E-8   &
       1.44535937822583 &         2.69368186E-8 \\\hline
6\%  & 1.44534387291222 &2.50578633E-8  &
       1.44534393955980   &     2.50203139E-8 \\\hline
\end{tabular}
\end{center}
\caption{Effective index of the PCF shown in Fig. \ref{hexagonring}(b).}
\label{tab:NumericalResult3}
\end{table}
We have also studied the convergence rates for these two PCFs. The results are presented
in Fig.~\ref{Convergencefig2}, which shows the rapid convergence of our method. 
\begin{figure}[ht]  
\centering  
    \begin{subfigure}[t]{.45\linewidth}
        \centering
        \includegraphics[width=1.1\linewidth]{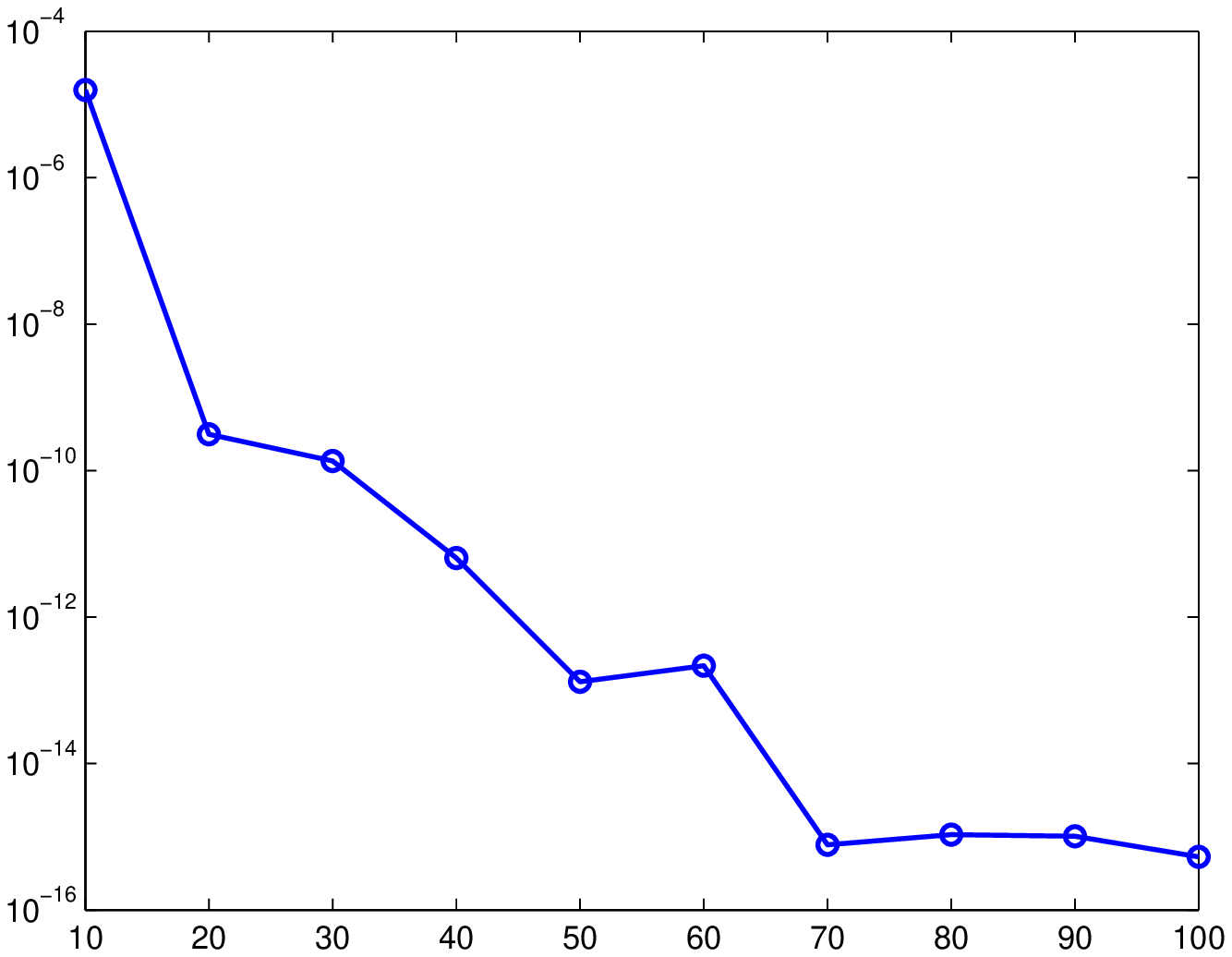}
        \caption{}
    \end{subfigure}
    \quad
    \begin{subfigure}[t]{.45\linewidth}
        \centering
        \includegraphics[width=1.1\linewidth]{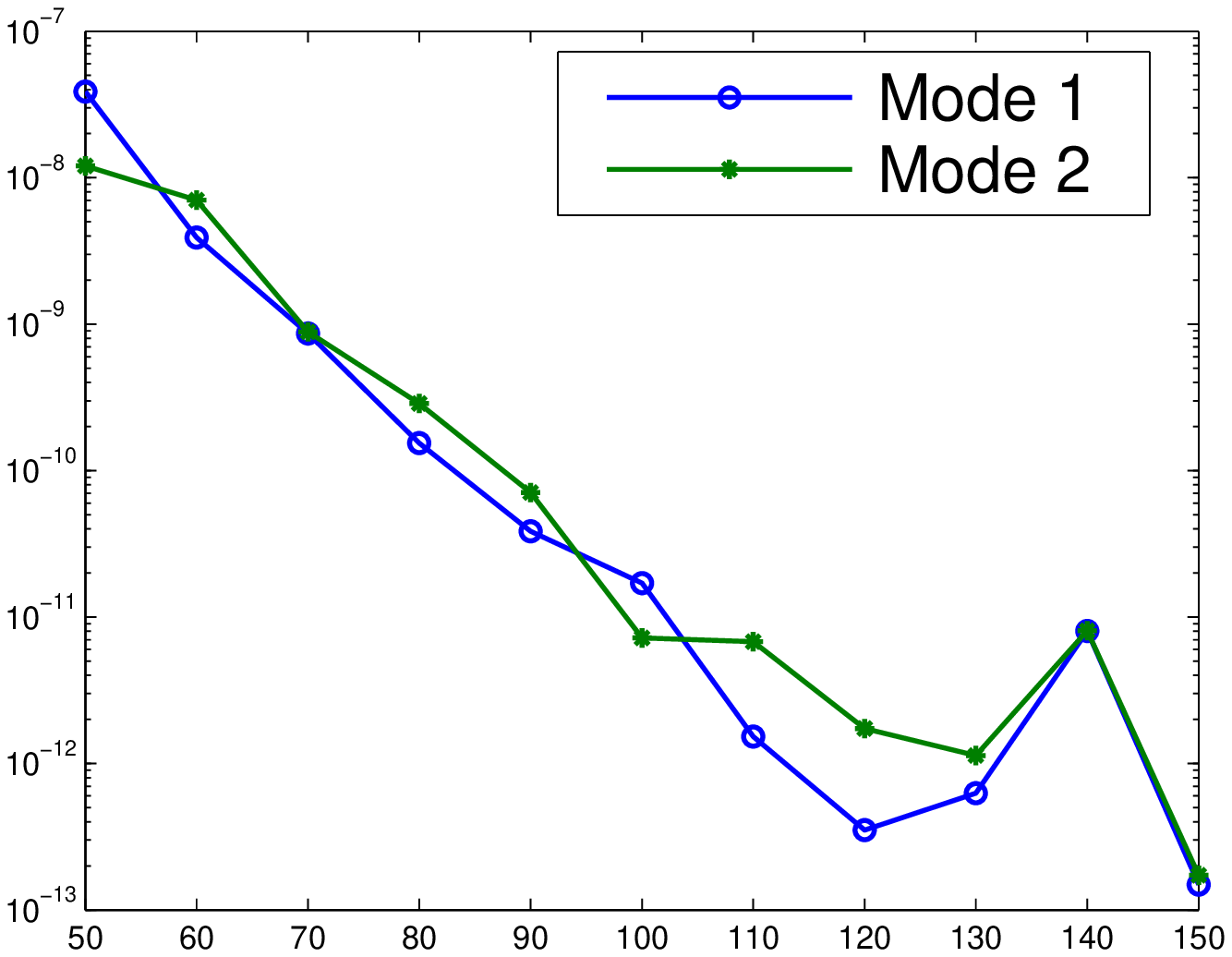}
        \caption{}
    \end{subfigure}
    \caption{Convergence study of the PCFs shown in Fig. \ref{hexagonring}.
      The $x$ axis shows the number of discretization
      points for each hole and the $y$ axis shows the relative error of the computed effective index. The reference
      value is obtained with $200$ points on each hole boundary.}
    \label{Convergencefig2}
\end{figure} 

\subsection{Example 3: PCFs with multiple-layer hexagon rings}
We now consider PCFs with multiple layers. The examples are taken from \cite{lu1,pone}.
The first PCF has
five layers of hexagon rings surrounding a circular core as shown in Figure
\ref{hexagonring2}(a). The small air holes are equally spaced with the
hole pitch $\Lambda = 2.74\mu m$ and the diameter $d=0.95\Lambda $. The diameter
of the core is $l= 2.5d$. The refractive indices of the glass surroundings and
the air hole are $1.45$ and $1$, respectively. The wavelength of the incident field
is $1510 nm$. The second PCF has an elliptic core with minor axis $2a = 2.3\mu m $
and major axis $2b = 4.6 \mu m$ surrounded by a three layers of hexagon rings.
The hole pitch of the small air holes is $\Lambda=2\mu m$ and the diameter of
each hole is $d=0.9\Lambda$. The incident wavelength for the second PCF is $\lambda = 1420 nm$.
These two PCFs are shown in Fig. \ref{hexagonring2}.
\begin{figure}[htbp]
    \centering
    \begin{subfigure}[t]{.45\linewidth}
        \centering
        \includegraphics[width=1.0\linewidth]{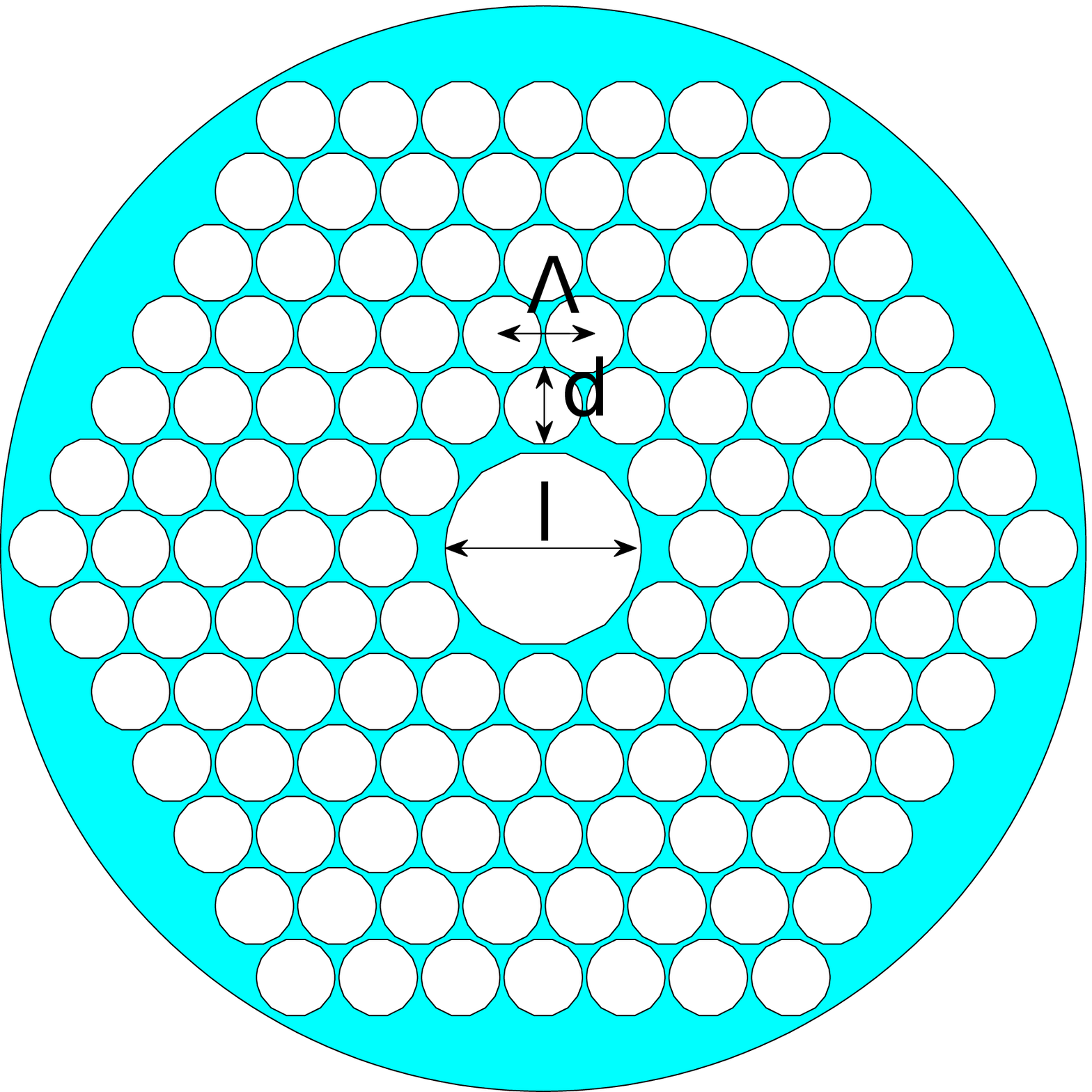}
        \caption{}
    \end{subfigure}
    \quad
    \begin{subfigure}[t]{.45\linewidth}
        \centering
        \includegraphics[width=1.0\linewidth]{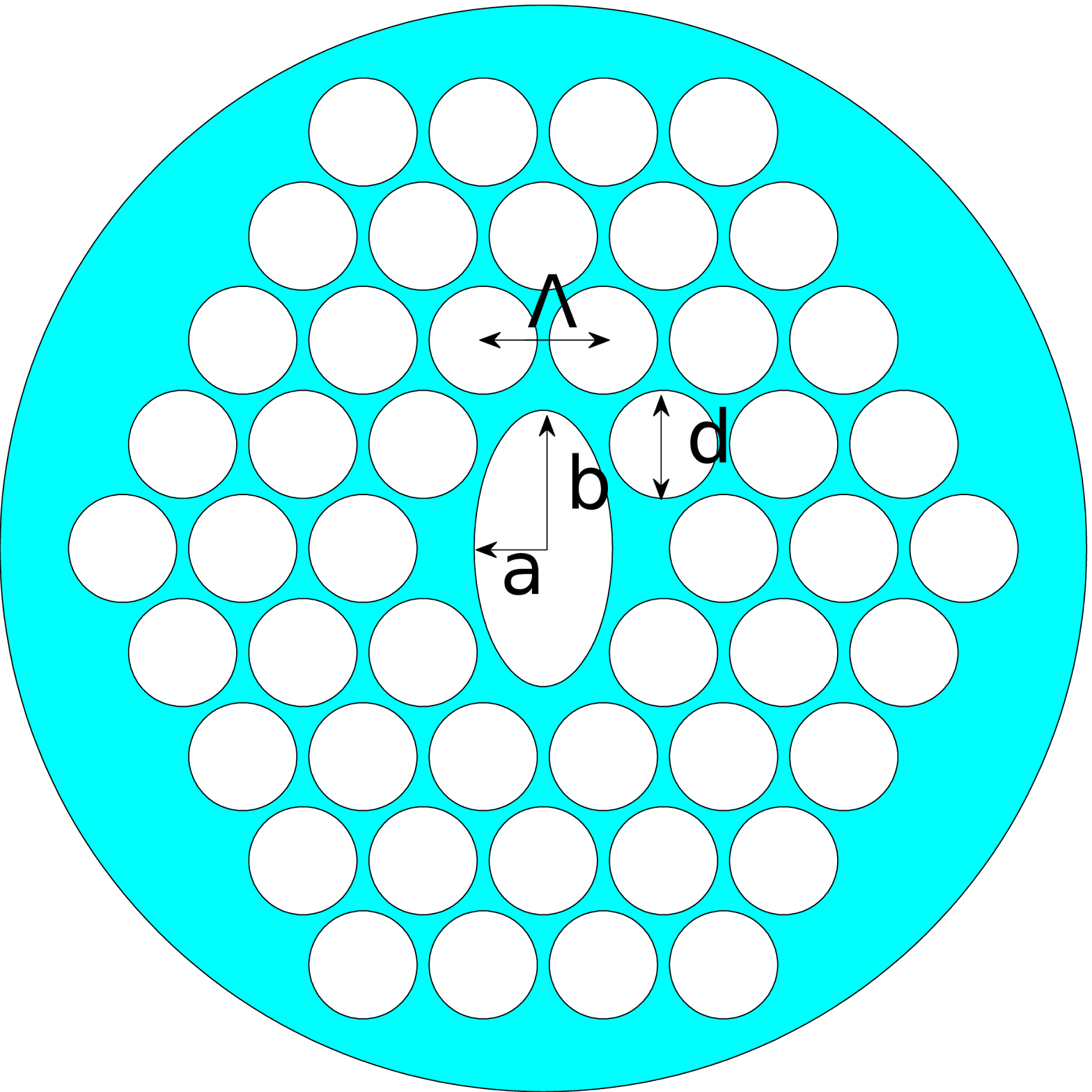}
        \caption{}
    \end{subfigure}
    \caption{Geometry of two PCFs. For the PCF shown in the left panel, parameters are given by
      $\Lambda = 2.74 \mu m$, $d = 0.95\Lambda$, $l= 2.5d$. For the PCF shown in the right panel (b),
      parameters are given by $\Lambda=2\mu m$, $d = 0.9\Lambda$, $a=1.15\mu m$, $b=2.3\mu m$.}
    \label{hexagonring2}
\end{figure}      

The results of our computation for these two PCFs are shown in Table \ref{tab:NumericalResult4}.
For comparison, we also list the corresponding results in \cite{lu1} in columns 4 and 5. The results
in \cite{lu1} are seen to be very accurate and there is a $13$ digit agreement between our results
and the results in \cite{lu1}. For the first PCF shown in Fig. \ref{hexagonring2}(a), we discretize each
small hole with $100$ points and the center hole with $200$ points; while for the second PCF shown in Fig.
\ref{hexagonring2}(b), we discretize each circular hole with $120$ points and the center
ellipse with $240$ points.
\begin{table}[htbp]
\begin{center}
\begin{tabular}{|c|c|c|c|c|} \hline
   & \multicolumn{2}{|c|}{$n_e$}      &  \multicolumn{2}{|c|}{$n_e$ in \cite{lu1}} \\\hline
& Real& Imaginary &  Real& Imaginary \\\hline
  \multirow{1}{*}{4(a)} & 0.9845160008345&3.41146823E-8 & 0.984516000835&3.41147E-8 \\\hline
 \multirow{2}{*}{4(b)} & 0.9390335474112 & 6.7418067299E-4  & 0.9390335474115 & 6.741806730E-4 \\\cline{2-5}
  & 0.9381625147574  & 2.2133063780E-3& 0.9381625147578   & 2.213306378E-3\\\hline
\end{tabular}
\end{center}
\caption{Effective index of the PCFs shown in Fig. \ref{hexagonring2}. The second and third columns list
  the values obtained by our method and the last
  two columns list the values in \cite{lu1}. 
}
\label{tab:NumericalResult4}
\end{table}
\section{Rectangular waveguides}
In this section, we present a detailed numerical study
on a high refractive index contrast silica waveguide. The cross section of the
waveguide is of the square shape with the side length equal to $3.4\mu m$.
The refractive index of the cladding is $n_0=1.4447$, while that of the core
is $2\%$ higher, i.e., $n_1=1.4447\times 1.02$. The wavelength of the incident
field is $1550nm$.
\subsection{Accuracy and conditioning of the SKIE formulation}
We first check the accuracy and conditioning of the SKIE formulation. For this,
we will solve the linear system $M(n_e)x = b$ with a nonzero right hand side
for some $n_e\in [n_0, n_1]$. We set $n_e = 1.451$ for testing purpose and we
construct artificial electromagnetic field by placing a point source outside
the square for the interior field and a point source inside the square for the
exterior field. We then obtain the right hand side vector $b$ by simply computing
the difference of the tangential components of these artificial electromagnetic
fields. After we have solved the linear system, we may use the representation
\eqref{mullere}-\eqref{mullerh} to evaluate the electromagnetic field both inside
and outside and compare the numerical result with the exact solution (which
is the artificial electromagnetic field by construction). We use GMRES to solve
the linear system and GMRES terminates when the relative residual is less than
$10^{-14}$.
\begin{table}[htbp]
\begin{center}
\begin{tabular}{|c|c|c|c|c|c|} \hline
   $N$            & 150 &  300 & 450 & 600 & 750 \\ \hline
   $N_{\text{iter}}$& 32  &   32 &  32 &  34 &  34 \\ \hline
\end{tabular}
\end{center}
\caption{Conditioning study of the SKIE formulation for the rectangular waveguide.}
\label{recttab1}
\end{table}
In Table \ref{recttab1}, the first row lists the number of discretization points on
each side of the square, while the second row lists the number of iterations needed in
GMRES. We observe that the number of iterations in GMRES is almost independent of the
size of the linear system, which is a characteristics of the SKIE formulation.

\begin{figure}[htbp]
  \centering
  \includegraphics[width=0.6\linewidth]{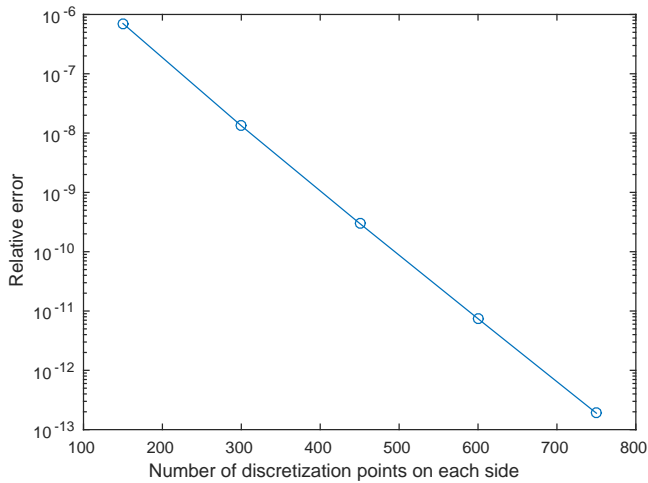}
  \caption{Convergence study of the rectangular waveguide. x-axis shows the number
    of discretization points on each side of the square; y-axis shows the relative
    error (in logarithmic scale) of the computed electromagnetic field against the
    exact value.
  }
\label{rectfig1}
\end{figure}

Figure \ref{rectfig1}
shows the relative error with various number of discretization points for each
side of the square. We observe that the order of convergence is about $10$, which
is in agreement with the theoretical value since the number of collocation points
on each chunk is set to $p=10$.

\subsection{Computation of the effective index}
It is known that this waveguide admits a single
propagation constant (or effective index) with double degeneracy. The
method in \cite{marcatili} produces an effective index approximately
equal to $1.458\cdots$. For comparison purpose, we have also implemented
the formulation in \cite{srep,pone} which uses four distinct single layer
potentials for the electromagnetic fields. 

\begin{table}[htbp]
\begin{center}
\begin{tabular}{|c|c|c|c|c|} \hline
  & \multicolumn{2}{|c|}{$n_e$ by SREP} &  \multicolumn{2}{|c|}{$n_e$ by SKIE} \\
  \hline
 N & Real& Imaginary &  Real& Imaginary \\\hline
150  & 1.45867110663907 &3E-10   &
       1.45860141500175  &      3E-9 \\\hline
300 &  1.45884317112249  &-2E-10   &
       1.45860141488787  &      6E-11 \\\hline
450 &  1.45890000315967 & 9E-14   & 
       1.45860141488572      &  1E-12 \\\hline
600  & 1.45887005720651 & -7E-6 &
       1.45860141488567   &      3E-14 \\\hline
\end{tabular}
\end{center}
\caption{Effective index of the rectangular waveguide. The first column lists
  the number of discretization points on each side of the square; the second and
  third columns list the real and imaginary parts of the effective index found via
  the formulation in \cite{srep,pone}; the fourth and fifth columns list the real
  and imaginary parts of the effective index found via our SKIE formulation.}
\label{recttab2}
\end{table}

Table \ref{recttab2} shows
the effective index $n_e$ found by the formulation in \cite{srep,pone} (denoted
by SREP in the table) and by our SKIE formulation for various number of
discretization points.
We observe that while the SKIE formulation exhibits
a consistent convergence behavior to about $13$ digit accuracy as $N$ increases,
the single layer representation behaves much more erratically and achieves
about only $4$ digit accuracy. This is because
that the resulting matrix from the discretization of the single layer representation
becomes more and more ill-conditioned as $N$ increases even when $n_e$ is
sufficiently far away from the root of the function defined in
\eqref{mullerfunction}.

\begin{table}[htbp]
\begin{center}
\begin{tabular}{|c|c|c|c|c|} \hline
   $N$                  &   150 &     300 &    450 &    600 \\ \hline
   $\kappa_{\text{SREP}}$ & 4.8E+5 & 3.4E+7 & 2.0E+9 & 1.0E+10 \\ \hline
\end{tabular}
\end{center}
\caption{Condition number of the formulation in \cite{srep,pone}
  for the rectangular waveguide.
  The first row lists the number of discretization points on each side of the
  square, while the second row lists the condition number of the resulting matrix.
  $n_e$ is set to $1.451$ as in Table \ref{recttab1}.}
\label{recttab3}
\end{table}
Table \ref{recttab3}
lists the condition numbers of the matrix $M_{\text{SREP}}$ for various
number of discretization points, which shows that the condition number of the matrix
increases very rapidly even when $n_e$ is sufficiently far away from the
root of the function in \eqref{mullerfunction}.
\begin{figure}[htbp]
  \centering
  \includegraphics[width=\linewidth]{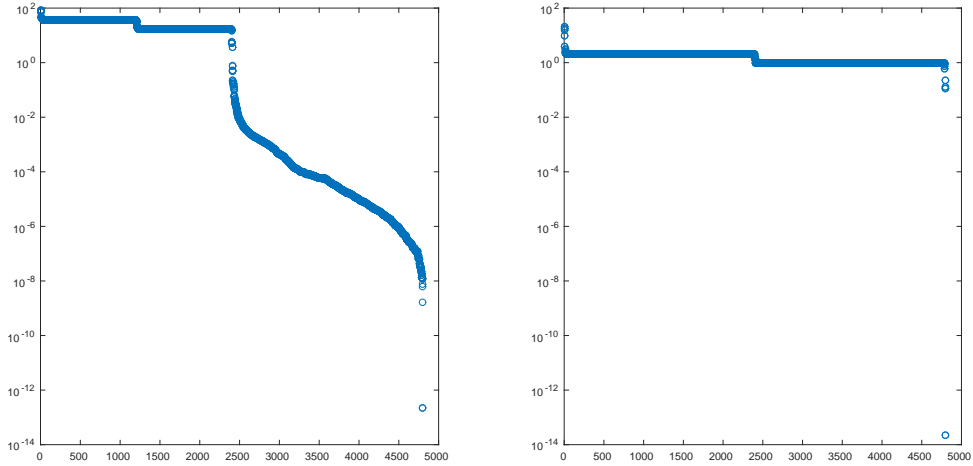}
  \caption{Singular value distribution of the matrices. The left panel shows
    the singular values of $M_{\text{SREP}}(n_e)$ when the size of the matrix is
    $4800$ and $n_e =1.45884317112249-2E-10i$; while the right panel shows
    the singular values of $M_{\text{SKIE}}(n_e)$ when the size of the matrix is
    $4800$ and $n_e=1.45860141488787 +6E-11i$.   
  }
\label{rectfig2}
\end{figure}

We have also carried out the singular value decomposition (SVD) for the
matrices $M_{\text{SREP}}$ and $M_{\text{SKIE}}$ at the effective indices
listed in Table \ref{recttab2} for $N=300$. The singular values are plotted
out in Figure \ref{rectfig2}. We observe that the SKIE formulation
has much cleaner singular value distribution. Indeed, the smallest two singular
values are both about $2.2\times 10^{-14}$, while the third smallest singular value
is about $0.1096$. On the other hand, the singular values of $M_{\text{SREP}}$
decreases almost continuously; the smallest two singular values are about
$2\times 10^{-13}$, while the third smallest singular value is about
$1.6\times 10^{-9}$. To sum up, the SKIE formulation enables us to find
the effective index accurately and robustly; while non-SKIE formulation
will either give low accuracy or lead to spurious propagation mode due to
ill-conditioning.

\section{Conclusions}
We have constructed a second kind integral equation formulation for the
mode calculation of optical waveguides or fibers. The resulting numerical
algorithm is capable of finding the propagation modes of optical waveguides
or fibers with an arbitrary number of cores or holes of arbitrary shape (smooth
or with corners). The algorithm is high order accurate so that it is capable
of computing the propagation constant (including its imaginary part which is
related to the propagation loss of the mode) with high fidelity.
The algorithm is robust and well-conditioned due to its SKIE formulation.
The algorithm is efficient since one only needs to discretize the material
interfaces. This enables practitioners in the integrated photonics industry
to have a reliable simulation tool for designing more compact and efficient
optical components or devices.
%
%
%
\section*{Acknowledgements}
The authors would like to thank Prof. Leslie Greengard at Courant Institute for helpful discussions.
\bibliographystyle{siam}
\bibliography{journalnames,fmm,photonics,qbx,reference_lj}

\end{document}